\documentclass{amsart}
\usepackage{amssymb, amsmath, amsthm, graphics, comment, xspace, enumerate}
\usepackage{hyperref}
\usepackage{color}
\newtheorem{thm}{Theorem}[section]
\newtheorem{cor}[thm]{Corollary}
\newtheorem{lem}[thm]{Lemma}

\newtheorem{prop}[thm]{Proposition}
\theoremstyle{definition}
\newtheorem{defn}[thm]{Definition}
\newtheorem{example}[thm]{Example}
\theoremstyle{remark}
\newtheorem{rem}[thm]{Remark}
\numberwithin{equation}{section}
\begin{document}
\title[Uniqueness sequences for multidimensional...]{Uniqueness sequences for multidimensional vector-valued Laplace transform}

\author{Marko Kosti\' c}
\address{Faculty of Technical Sciences,
University of Novi Sad,
Trg D. Obradovi\' ca 6, 21125 Novi Sad, Serbia}
\email{marco.s@verat.net}

{\renewcommand{\thefootnote}{} \footnote{
\\ \text{  }  \ \    {\it Key words and phrases.} Multidimensional vector-valued Laplace transform, uniqueness sequences, subordination principles.}}

\begin{abstract}
In this research article, we consider the uniqueness sequences for multidimensional vector-valued Laplace transform. We establish the fundamental relationships between uniqueness sequences for one-dimensional Laplace transform and uniqueness sequences for multidimensional Laplace transform. We also provide several illustrative examples, open problems and useful observations in the above direction.
\end{abstract} 

\maketitle 

\section{Introduction and preliminaries}

It seems that the modern form of 
Laplace transform was probably used for the first time by L. Euler in 1744. Concerning the Laplace transform and its applications, we can warmly recommend reading the research monographs \cite{a43} by W. Arendt et al., \cite{debnathb} by
L. Debnath and D. Bhatta, \cite{dech} by G. Doetsch, \cite{knjigaho}-\cite{FKP}
by M. Kosti\' c, \cite{prus} by J. Pr\"uss, \cite{widder} by
D. V. Widder and \cite{x263}
by
T.-J. Xiao and J. Liang.

The double Laplace transform of scalar-valued functions was initially analyzed  by D. L. Bernstein \cite{bern0} (1939), J. C. Jaeger \cite{jager} (1939--1941) and L. Amerio \cite{amerio} (1940). In a series of our recent research studies, we have investigated the multidimensional Laplace transform of functions with values in  sequentially complete locally convex spaces and their applications to the abstract Volterra integro-differential inclusions with multiple variables (\cite{vm,glfc,mvlt}).

The uniqueness, inversion and approximations of vector-valued Laplace transform have been seriously investigated by B. B\"aumer, C. Mihai, F. Neubrander and their collaborators; see \cite{b1}-\cite{b5}, \cite{kim}, \cite{mihai} and the references quoted therein. The notion of a uniqueness sequence for the vector-valued Laplace transform has been essentially employed by B. B\"aumer in \cite{b2}, where the author has considered the regularized solutions to the backwards heat equation in the space $L^{2}[0,\pi];$ cf. also \cite[Example 2.8.1, Example 2.8.2]{knjigah} for some later contributions made. 

The main results concerning uniqueness sequences for Laplace transform were proved by Yu-Cheng Shen in 1947 (\cite{shen}) and C. Mihai in 2009 (\cite{mihai}); cf. also \cite[Theorem 1.11.1]{a43}, \cite[Lemma 1.2]{b5} and \cite[pp. 72--79]{dech}:

\begin{thm}\label{ibeer}
\begin{itemize}
\item[(i)] Suppose that $(\lambda_{k})_{k\in {\mathbb N}}$ is an infinite sequence without accumulation points such that $\Re \lambda_{k}\geq \delta>0,$ $k\in {\mathbb N}$ for some $\delta>0$, and
\begin{align}\label{ret}
\sum_{k=1}^{+\infty}\Biggl[ 1-\Biggl| \frac{  \lambda_{k}-1 }{  \lambda_{k}+1 }\Biggr|\Biggr]=+\infty.
\end{align}
Then $(\lambda_{k})_{k\in {\mathbb N}}$ is a uniqueness sequence for Laplace transform.
\item[(ii)] Suppose that $(\lambda_{k})_{k\in {\mathbb N}}$ is a complex sequence without accumulation points, $\Re \lambda_{k}\geq \delta>0,$ $k\in {\mathbb N}$ for some $\delta>0$, the sum in \eqref{ret} is finite and $|\arg(\lambda_{k})|\leq \theta $, $k\in {\mathbb N}$ for some $\theta \in (0,\pi/2).$ Then $(\lambda_{k})_{k\in {\mathbb N}}$ is not a uniqueness sequence for Laplace transform.
\item[(iii)] Suppose that $(\lambda_{k})_{k\in {\mathbb N}}$ is a complex sequence without accumulation points, $\Re \lambda_{k}\geq \delta>0,$ $k\in {\mathbb N}$ for some $\delta>0$ and $|\arg(\lambda_{k})|\leq \theta $, $k\in {\mathbb N}$ for some $\theta \in (0,\pi/2).$ Then $(\lambda_{k})_{k\in {\mathbb N}}$ is a uniqueness sequence for Laplace transform if and only if \eqref{ret} holds.
\end{itemize}
\end{thm}

In this research article, we analyze the uniqueness sequences for multidimensional vector-valued Laplace transform. First of all, we would like to note that the proof of Theorem \ref{ibeer} leans heavily on the use of Jensen's formula, which  asserts that the zeroes $(\alpha_{k})_{k\in {\mathbb N}}$ of a bounded analytic function $h: \{z\in {\mathbb C} : |z|<1\} \rightarrow {\mathbb C},$ must satisfy the Blaschke condition $\sum_{k=1}^{+\infty}(1-|\alpha_{k}|)<+\infty .$ Unfortunately, there is no satisfactory extension of Jensen's formula to the functions defined in polydiscs and it would be very difficult to provide an analogue of Theorem \ref{ibeer} in the higher-dimensional setting (cf. \cite[pp. 300--311]{rudin1} and \cite[pp. 61--64]{rudin2} for more details). 

First of all, we recall the basic definitions and results about multidimensional vector-valued Laplace transform in Subsection \ref{suba1}. The notion of a uniqueness sequence for multidimensional vector-valued Laplace transform is introduced in Definition \ref{mila}, which is given at the very beginning of Section \ref{us}. The main aim of Theorem \ref{gade} is to extend the statements of \cite[Lemma 1.7.1, Proposition 1.7.2]{a43} to the multi-dimensional setting as well as to provide some concrete examples of uniqueness sequences for multidimensional Laplace transform. If $\Re \lambda_{k}^{j} >0$ for all $k\in {\mathbb N}$ and $j\in {\mathbb N}_{n},$ where $ {\mathbb N}_{n}:= \{1,...,n\},$ then Proposition \ref{total} claims that the set of functions
$\{\exp(-\lambda_{k}^{1}t_{1}-...-\lambda_{k}^{n}t_{n}) : k\in {\mathbb N}\}$ is total in $L^{1}([0,+\infty)^{n})$ if and only if the sequence $((\lambda_{k}^{1},...,\lambda_{k}^{n}))_{k\in {\mathbb N}}$ is a uniqueness sequence for the multidimensional Laplace transform. In Theorem \ref{tor} and Proposition \ref{zxc}, we elucidate a fundamental relationship between the uniqueness sequences for one-dimensional Laplace transform and the uniqueness sequences for multidimensional Laplace transform; cf. also Example \ref{or} and Example \ref{oro}. We prove that any uniqueness sequence for multidimensional Laplace transform which belongs to the set $({\mathbb C}_{+})^{n},$ where ${\mathbb C}_{+}:=\{z\in {\mathbb C} :  \Re z>0\},$ has to be a subsequence of $((\lambda_{k_{1}}^{1},...,\lambda_{k_{n}}^{n}))_{  (k_{1},...,k_{n})\in {\mathbb N}^{n}}$, where $(\lambda_{k}^{1})_{k\in {\mathbb N}},...,$ and $ (\lambda_{k}^{n})_{k\in {\mathbb N}}$ are uniqueness sequences for one-dimensional Laplace transform as well as that the sequence $((\lambda_{k_{1}}^{1},...,\lambda_{k_{n}}^{n}))_{  (k_{1},...,k_{n})\in {\mathbb N}^{n}}$ itself is a uniqueness sequence for multidimensional Laplace transform. The following problem is essentially important in our study:\vspace{0.1cm} 

\noindent {\bf Problem 1.} Suppose that 
$(\lambda_{k}^{1})_{k\in {\mathbb N}},...,$ and $ (\lambda_{k}^{n})_{k\in {\mathbb N}}$ are uniqueness sequences for Laplace transform, and  $\Re \lambda_{k}^{j} >0$ for all $k\in {\mathbb N}$ and $j\in {\mathbb N}_{n}$. Could we clarify some non-trivial conditions ensuring that a proper subsequence of $((\lambda_{k_{1}}^{1},...,\lambda_{k_{n}}^{n}))_{  (k_{1},...,k_{n})\in {\mathbb N}^{n}}$ is a uniqueness sequence for multidimensional Laplace transform?\vspace{0.1cm}

The M\"untz approximation theorem states that the sequence of functions $\{t^{\lambda_{k}} : k\in {\mathbb N}_{0}\},$ where $0=\lambda_{0}<\lambda_{1}<...<\lambda_{k}<...,$ is total in $C([0,1])$ if and only if $\sum_{k=1}^{+\infty}\lambda_{k}^{-1}=+\infty.$ The analysis of multidimensional M\"untz approximation type theorems is rather non-trivial and we can recommend reading the research articles \cite{bloom} by T. Bloom, \cite{heller} by S. Hellerstein, \cite{kroo} by  A. Kro\'o, \cite{ogawa} by S. Ogawa, K. Kitahara, \cite{ronkin} by L. I. Ronkin and references cited therein for more details about this issue.
In Theorem \ref{kunja}, we will provide a partial answer to Problem 1 by using a multivariate M\"untz theorem established by  L. I. Ronkin in \cite[Theorem 1]{ronkin}. 

Subsection \ref{subasub} investigates the subordination principles for uniqueness sequences; our main result in this direction are Theorem \ref{ressuba} and Corollary \ref{ojha}.
Finally, we would like to say that it would be very tempting to find some relevant applications of uniqueness sequences for multidimensional Laplace transform to abstract Volterra integro-differential inclusions with multiple variables. 

\subsection{Multidimensional vector-valued Laplace transform}\label{suba1} 

Let $n\in {\mathbb N},$ let $X$ be a non-trivial Hausdorff sequentially complete locally convex space\index{sequentially complete locally convex space!Hausdorff} over the field of complex numbers (SCLCS), and let
the abbreviation $\circledast$ designate the fundamental system of seminorms\index{system of seminorms} which defines the topology of $X.$  For further information concerning the integration of functions with values in SCLCSs, we refer the reader to \cite{FKP}.

We recall the following notion from \cite{mvlt}:

\begin{defn}\label{svi}
Suppose that $X$ is an SCLCS, $f: [0,+\infty)^{n}\rightarrow X$ is a locally integrable function and $(\lambda_{1},...,\lambda_{n})\in {\mathbb C}^{n}$. If 
\begin{align*} 
F\bigl(\lambda_{1},...,\lambda_{n}\bigr)&:=\lim_{t_{1}\rightarrow +\infty;...;t_{n}\rightarrow +\infty}\int^{t_{1}}_{0}...\int^{t_{n}}_{0}e^{-\lambda_{1}s_{1}-...-\lambda_{n}s_{n}}f\bigl(s_{1},...,s_{n}\bigr)\, ds_{1}\, ...\, ds_{n}
\\& :=\int^{+\infty}_{0}...\int^{+\infty}_{0}e^{-\lambda_{1}t_{1}-...-\lambda_{n}t_{n}}f\bigl(t_{1},...,t_{n}\bigr)\, dt_{1}\, ...\, dt_{n},
\end{align*}
exists in topology of $X,$ 
then we say that the Laplace integral $({\mathcal Lf})(\lambda_{1},...,\lambda_{n}):=\hat{f}(\lambda_{1},...,\lambda_{n}):=F(\lambda_{1},...,\lambda_{n})$ exists.
We define the region of convergence of Laplace integral $
\Omega(f)$ by
$
\Omega(f):=\{ (\lambda_{1},...,\lambda_{n})\in {\mathbb C}^{n} : F(\lambda_{1},...,\lambda_{n})\mbox{ exists}\}.
$
Furthermore, if
for each seminorm $p\in \circledast$ we have
\begin{align*}
\int^{+\infty}_{0}...\int^{+\infty}_{0}p\Bigl(e^{-\lambda_{1}t_{1}-...-\lambda_{n}t_{n}}f\bigl(t_{1},...,t_{n}\bigr)\Bigr)\, dt_{1}\, ...\, dt_{n}<+\infty,
\end{align*}
then we say that the Laplace integral $F(\lambda_{1},...,\lambda_{n})$
converges absolutely. We define the region of absolute convergence of Laplace integral $
\Omega_{abs}(f)$ by
$
\Omega_{abs}(f):=\{ (\lambda_{1},...,\lambda_{n})\in {\mathbb C}^{n} : F(\lambda_{1},...,\lambda_{n})\mbox{ converges absolutely}\}.
$
Finally, we define $\Omega_{b}(f)$ as the set of all tuples $(\lambda_{1},...\lambda_{n})\in {\mathbb C}^{n}$ such that the set{\small
\begin{align*} 
\Biggl\{
\int^{t_{1}}_{0}\int^{t_{2}}_{0}...\int^{t_{n}}_{0}e^{-\lambda_{1}s_{1}-\lambda_{2}s_{2}-...-\lambda_{n}s_{n}}f\bigl( s_{1},s_{2},...,s_{n}\bigr)\, ds_{1} \, ds_{2}...\, ds_{n} :  t_{1}\geq 0,...,t_{n}\geq 0\Biggr\}
\end{align*}}
is bounded in $X.$
\end{defn}

A locally integrable function $f: [0,+\infty)^{n}\rightarrow X$ is said to be Laplace transformable if there exist real numbers $\omega_{i}\in {\mathbb R}$ ($1\leq i\leq n$) such that $ \{\lambda_{1}\in {\mathbb C} : \Re \lambda_{1}>\omega_{1}\} \times ... \times \{\lambda_{n}\in {\mathbb C} : \Re \lambda_{n}>\omega_{n}\} \subseteq \Omega(f).$ If $f: [0,+\infty)^{n}\rightarrow X$ is locally integrable, $\emptyset \neq \Omega \subseteq \Omega(f) \cap \Omega_{b}(f)$ is open and $(\lambda_{1},...,\lambda_{n})\in \Omega,$ then the mapping $F: \Omega \rightarrow X$ is holomorphic;
for the basic information concerning holomorphic vector-valued functions of several variables, we refer the reader to the research article \cite{kruse} by K. Kruse and list of references quoted in \cite{mvlt}.  

If $u\in L_{loc}^{1}([0,\infty)^{n})$ and $a\in L_{loc}^{1}([0,\infty)^{n})$, then we define 
\begin{align*}                    
\bigl(a\ast_{0}u\bigr)(t) := \int^{t_{1}}_{0}\cdot ....\cdot \int^{t_{n}}_{0}a\bigl(t_{1}-s_{1},...,t_{n}-s_{n}\bigr)u\bigl(s_{1},...,s_{n}\bigr)\, ds_{1} ...\, ds_{n},
\end{align*}
for any $ t= (t_{1},...,t_{n}\bigr)\in [0,+\infty)^{n}.$ The following generalization of finite convolution product $\ast_{0}$ has recently been introduced and analyzed in \cite{new}:
Suppose that $1\leq l\leq n$, $1\leq j\leq {n\choose l}$ and $D_{l,j}=\{j_{1},...,j_{l}\}$ is a fixed subset of ${\mathbb N}_{n}$, where $1\leq j_{1}<...<j_{l}\leq n.$ If $a_{l,j}\in L_{loc}^{1}([0,+\infty)^{l})$, then we define 
\begin{align} \notag &
\Bigl( a_{l,j}\ast_{0}^{l,j}u\Bigr)\bigl(t_{1},...,t_{n}\bigr):=\int^{t_{j_{1}}}_{0}...\int^{t_{j_{l}}}_{0}a_{l,j}\bigl( t_{j_{1}}-s_{j_{1}},..., t_{j_{l}}-s_{j_{l}}\bigr)
\\  \label{mire} & \times u\Bigl(t_{1},...,t_{j_{1}-1},s_{j_{1}},t_{j_{1}+1},...,t_{j_{2}-1},s_{j_{2}},t_{j_{2}+1},...,t_{j_{l}-1}, s_{j_{l}},t_{j_{l}+1},...,t_{n}\Bigr)\, ds_{j_{1}}...\, ds_{j_{l}},
\end{align}
for any $(t_{1},...,t_{n}) \in [0,+\infty)^{n}.$ 
 In Fr\' echet spaces, the convolution $(a_{l,j}\ast_{0}^{l,j}u)(\cdot)$ exists if $u\in L_{loc}^{1}([0,+\infty)^{n}:X),$ while in general SCLCSs the convolution $(a_{l,j}\ast_{0}^{l,j}u)(\cdot)$ exists if, in addition to the above, $a_{l,j}\in C([0,+\infty)^{l} )$ or $u\in C([0,+\infty)^{n} : X );$ in the first case, we have $a_{l,j}\ast_{0}^{l,j}u\in L_{loc}^{1}([0,+\infty)^{n}:X)$, while, in the second case, we have $a_{l,j}\ast_{0}^{l,j}u\in C([0,+\infty)^{n}:X).$

We need the following result from \cite{new}, as well:

\begin{lem}\label{nova}
Suppose that $0\leq l\leq n,$ $1\leq j\leq {n\choose l}$, $a_{l,j}\in L_{loc}^1([0,+\infty)^{l})$, $u\in L_{loc}^1([0,+\infty)^{n}:X)$ and  $\Omega =\{\lambda \in {\mathbb C} : \Re \lambda>\omega_{1}\} \times ... \times \{\lambda \in {\mathbb C} : \Re \lambda>\omega_{n}\}\subseteq   \Omega_{abs}(u).$
Then
the following holds ($\lambda=( \lambda_{1},...,\lambda_{n})$):
\item[(i)] Let $X$ be a Fr\' echet space. Then $(a_{l,j}\ast_{0}^{l,j}u)(\cdot)\in L_{loc}^{1}([0,+\infty)^{n}:X)$ and  
\begin{align*}
\Bigl({\mathcal L}\bigl(a_{l,j}\ast_{0}^{l,j}u\bigr)\Bigr)( \lambda)=\widetilde{a_{l,j}}\bigl( \lambda_{j_{1}},...,  \lambda_{j_{l}} \bigr) \cdot \tilde{u}( \lambda),
\end{align*}
provided that \emph{(a)-(b)} hold, where:
\begin{itemize}
\item[(a)] 
$\Omega\subseteq \Omega_{abs}(a_{l,j}\ast_{0}^{l,j}u),$ and
\item[(b)] For every $t_{1},...,t_{j_{1}-1},t_{j_{1}+1},t_{j_{2}-1},t_{j_{2}+1},...,t_{j_{l}-1}, t_{j_{l}+1},...,t_{n}$ we have
\begin{align*} &
\bigl\{ \lambda_{j_{1}}\in {\mathbb C} : \Re \lambda_{j_{1}}>\omega_{j_{1}}\bigr\} \times ... \times \bigl\{ \lambda_{j_{l}}\in {\mathbb C} : \Re \lambda_{j_{l}}>\omega_{j_{l}}\bigr\}\subseteq
\\&  \Omega_{abs}\bigl(a_{l,j}\bigr)  \cap \Omega_{abs}\Bigl(u\bigl(t_{1},...,t_{j_{1}-1},\cdot,t_{j_{1}+1},...,t_{j_{2}-1},\cdot,t_{j_{2}+1},...,t_{j_{l}-1}, \cdot,t_{j_{l}+1},...,t_{n}\bigr)\Bigr).
\end{align*}
\end{itemize}
\end{lem}

\section{Uniqueness sequences}\label{us}

A sequence $(\lambda_{k})_{k\in {\mathbb N}}$ of complex numbers is said to be a uniqueness sequence for the Laplace transform if for any locally integrable function $f : [0,+\infty) \rightarrow X$ such that $\emph{abs}(f)<\Re \lambda_{k}$ and $\hat{f}(\lambda_{k})=0$ for all $k\in {\mathbb N}$, we have $f=0$ a.e. (cf. \cite[Remark 1.7.4]{a43} and \cite[pp. 57--58]{FKP} for the notion). Since we have assumed that the pivot space $X$ is non-trivial, we can simply prove with the help of Hanh-Banach theorem that the notion of uniqueness sequence for the Laplace transform
is meaningful and does not depend on the choice of space $X$. For example, the sequence $-i,\pm 2i,\pm 3i,...$ is not a uniqueness sequence for the Laplace transform since it annulates the Laplace transform of function $f(\cdot)$ given by $f(t):=e^{it},$ $0\leq t\leq 2\pi$ and $f(t):=0,$ $t>2\pi;$ see \cite[Bemerkungen, p. 76]{dech}.

Let us also observe that, in the above definition, we have assumed that $\emph{abs}(f)<\Re \lambda_{k}$ for all $k\in {\mathbb N}$ but not that $\hat{f}(\lambda_{k})$ exists for all $k\in {\mathbb N}$. In order to stay consistent with the notion introduced in the one-dimensional setting, we would like to propose the following definition:

\begin{defn}\label{mila}
A sequence $((\lambda_{k}^{1},...,\lambda_{k}^{n}))_{k\in {\mathbb N}}$ in ${\mathbb C}^{n}$ is said to be a uniqueness sequence for the multidimensional Laplace transform (in ${\mathbb R}^{n}$) if for any locally integrable function $f : [0,+\infty)^{n} \rightarrow X$ such that there exist numbers $\omega_{1} \in [-\infty,+\infty),...,$ $ \omega_{n}\in [-\infty,+\infty)$ with $\Omega:=\{\lambda_{1}\in {\mathbb C} : \Re \lambda_{1}>\omega_{1}\} \times ... \times \{\lambda_{n}\in {\mathbb C} : \Re \lambda_{n}>\omega_{n}\} \subseteq \Omega_{b}(f) \cap \Omega(f)$ and $\omega_{j}<\Re \lambda_{k}^{j}$ for all $k\in {\mathbb N}$ and $j\in {\mathbb N}_{n}$, the assumption $\hat{f}(\lambda_{k}^{1},...,\lambda_{k}^{n})=0$ for all $k\in {\mathbb N}$ implies $f=0$ a.e.
\end{defn}

As in the one-dimensional setting, the notion introduced in Definition \ref{mila} does not depend on the choice of space $X$; in the remainder of this paper, we will therefore assume that $X={\mathbb C}$. It can be simply shown that a finite sequence  in ${\mathbb C}^{n}$ cannot be a uniqueness sequence for the multidimensional Laplace transform as well as that, for every two sequences $(z_{k})_{k\in {\mathbb N}}$ and $(z_{k}')_{k\in {\mathbb N}}$ in ${\mathbb C}^{n}$, the assumptions $(z_{k})_{k\in {\mathbb N}}$ is a uniqueness sequence for the multidimensional Laplace transform and $\{z_{k} : k\in {\mathbb N} \} \subseteq \{z_{k}' : k\in {\mathbb N}\},$ imply that $(z_{k}')_{k\in {\mathbb N}}$ is a uniqueness sequence for the multidimensional Laplace transform; furthermore, any sequence obtained by dropping finitely many elements from a uniqueness sequence is likewise a uniqueness sequence for the multidimensional Laplace transform. In Definition \ref{mila}, we can equivalently assume that the region $\Omega$ belongs to the set $\Omega_{b}(f);$ cf.\cite[Theorem 2.7]{mvlt}.

In the sequel, we will analyze the uniqueness sequences 
$((\lambda_{k}^{1},...,\lambda_{k}^{n}))_{k\in {\mathbb N}}$ which satisfy that there exist real numbers $\omega_{1},...,\omega_{n}$ such that $ \Re \lambda_{k}^{j}>\omega_{j}$ for all $k\in {\mathbb N}$ and $j\in {\mathbb N}_{n}. $
If this is the case, then the following result enables us to reduce our further investigation to the case in which $\Re \lambda_{k}^{j}>0$ ($\Re \lambda_{k}^{j}\geq \delta$ for some real number $\delta>0$) for all $k\in {\mathbb N}$ and $j\in {\mathbb N}_{n}$; the proof is easy and can be left to the interested readers:

\begin{prop}\label{nmt}
Suppose that $(z_{1},...,z_{n}) \in {\mathbb C}^{n}$ and $((\lambda_{k}^{1},...,\lambda_{k}^{n}))_{k\in {\mathbb N}}$ is a sequence in ${\mathbb C}^{n}.$ Then $((\lambda_{k}^{1},...,\lambda_{k}^{n}))_{k\in {\mathbb N}}$ is a uniqueness sequence for the multidimensional Laplace transform if and only if $((\lambda_{k}^{1}+z_{1},...,\lambda_{k}^{n}+z_{n}))_{k\in {\mathbb N}}$ is a uniqueness sequence for the multidimensional Laplace transform.
\end{prop}

Now we will state and prove the folowing results:

\begin{prop}\label{gade}
Suppose that $a_{i}>0$ and $b_{i}>0$ for $1\leq i\leq n$.
\begin{itemize}
\item[(i)] Let $a=(a_{1},...,a_{n})$, $b= (b_{1},...,b_{n})$ and 
$$
e_{-\lambda_{k},a,b}\bigl(t_{1},...,t_{n}\bigr):=e^{-(a_{1}+k_{1}b_{1})t_{1}-...-(a_{n}+k_{n}b_{n})t_{n}},
$$ 
for any $t=(t_{1},...,t_{n})\in [0,+\infty)^{n}$ and $k=(k_{1},...,k_{n})\in {\mathbb N}_{0}^{n}.$
Then the set $\{e_{-\lambda_{k},a,b}(\cdot): k=(k_{1},...,k_{n})\in {\mathbb N}_{0}^{n}\}$ is total in $L^{1}([0,+\infty)^{n}).$
\item[(ii)] The sequence $((a_{1}+k_{1}b_{1},...,a_{n}+k_{n}b_{n}))_{ (k_{1},...,k_{n})\in {\mathbb N}_{0}^{n}}$ is a uniqueness sequence for the multidimensional Laplace transform.
\end{itemize}
\end{prop}

\begin{proof}
The Stone-Weierstrass theorem implies that the set of functions $\{a_{1}^{-1}\cdot .... \cdot a_{n}^{-1}t_{1}^{b_{1}k_{1}/a_{1}}\cdot ... \cdot t_{n}^{k_{n}b_{n}/a_{n}} : k=(k_{1},...,k_{n})\in {\mathbb N}_{0}^{n}\}$ is total in $C([0,1]^{n})$ and therefore in $L^{1}((0,1)^{n}).$ Now (i) follows from the fact that the mapping $\Phi : L^{1}((0,1)^{n}) \rightarrow L^{1}([0,+\infty)^{n})$, given by{\small
$$
\Bigl[\Phi (g)\Bigr]\bigl(t_{1},...,t_{n}\bigr):=a_{1} \cdot ... \cdot a_{n} e^{-a_{1}t_{1}-...-a_{n}t_{n}}g\bigl( e^{-a_{1}t_{1}},...,e^{-a_{n}t_{n}}\bigr),\ \bigl(t_{1},...,t_{n}\bigr) \in [0,+\infty)^{n} ,
$$}
for any $ g\in L^{1}((0,1)^{n}),$ is an isometric isomorphism which maps the function $a_{1}^{-1}\cdot .... \cdot a_{n}^{-1}t_{1}^{b_{1}k_{1}/a_{1}}\cdot ... \cdot t_{n}^{k_{n}b_{n}/a_{n}}$ to the function $e_{-\lambda_{k},a,b}(\cdot)$ for all $k\in {\mathbb N}_{0}^{n}.$ To prove (ii), suppose that a locally integrable function $f : [0,+\infty)^{n} \rightarrow X$ satisfies that there exist numbers $\omega_{1}' \in [-\infty,+\infty),...,$ $ \omega_{n}'\in [-\infty,+\infty)$ such that $\{\lambda_{1}\in {\mathbb C} : \Re \lambda_{1}>\omega_{1}'\} \times ... \times \{\lambda_{n}\in {\mathbb C} : \Re \lambda_{n}>\omega_{n}'\} \subseteq \Omega_{b}(f) $, $\omega_{j}'<\Re \lambda_{k}^{j}$ for all $k\in {\mathbb N}$, $j\in {\mathbb N}_{n},$ and $\hat{f}(\lambda_{k}^{1},...,\lambda_{k}^{n})=0$ for all $k\in {\mathbb N}.$ Define 
\begin{align}\label{kw}
G\bigl(t_{1},...,t_{n}\bigr):=\int^{t_{1}}_{0}\cdots \int^{t_{n}}_{0}f\bigl(s_{1},...,s_{n}\bigr)\, ds_{1}  ...  \, ds_{n},\quad \bigl(t_{1},...,t_{n} \bigr) \in [0,+\infty)^{n}.\end{align}
By \cite[Theorem 2.10(ii)]{mvlt}, there exists a real number $M\geq 1$ such that $|G(t_{1},...,t_{n})| \leq M\exp(\omega_{1}t_{1}+...+\omega_{n}t_{n}),$ $ (t_{1},...,t_{n} ) \in [0,+\infty)^{n},$ where $\omega_{j}:=\max(\omega_{j}',0)$ for $1\leq j\leq n.$ Keeping in mind \cite[Theorem 2.10(iii)]{mvlt}, it suffices to show that the assumption $\hat{G}\ (  a_{1}+k_{1}b_{1} ,...,a_{n}+k_{n}b_{n}  )=0$ for all $k=(k_{1},...,k_{n})\in {\mathbb N}_{0}^{n}$ with $k_{1}\geq k_{1}^{0},...,\ k_{n}\geq k_{n}^{0}$, where $a_{1}+k_{1}^{0}b_{1}>\omega_{1},...,\ a_{n}+k_{n}^{0}b_{n}>\omega_{n},$  implies 
$H (t_{1},...,t_{n} ) =0,$ $  (t_{1},...,t_{n} ) \in [0,+\infty)^{n},$ where $H (t_{1},...,t_{n} ) :=\exp(-\omega_{1}t_{1}-...-\omega_{n}t_{n})G(t_{1},...,t_{n}),$ $ (t_{1},...,t_{n} ) \in [0,+\infty)^{n}.$ Clearly, we have
$$
\hat{H}\bigl( \lambda_{1},...,\lambda_{n}\bigr)=\hat{G}\bigl(\lambda_{1}+\omega_{1},...,\lambda_{n}+\omega_{n}\bigr),\quad \lambda_{i}>0 \ \ (1\leq i\leq n),
$$
so that the prescribed assumption implies $\hat{H}\ (  a_{1}+k_{1}b_{1}-\omega_{1},...,a_{n}+k_{n}b_{n}-\omega_{n} )=0$ for all $k=(k_{1},...,k_{n})\in {\mathbb N}_{0}^{n}$ such that $k_{1}\geq k_{1}^{0},...,\ k_{n}\geq k_{n}^{0}$. Set $A_{1}:= a_{1}+k_{1}^{0}b_{1} -\omega_{1},...,\ A_{n}:= a_{n}+k_{n}^{0}b_{n} -\omega_{n}.$ Then $A_{1}>0,...,\ A_{n}>0$ and 
$$
\int^{+\infty}_{0}\cdots  \int^{+\infty}_{0}e^{-\bigl(A_{1}+k_{1}b_{1}\bigr)t_{1}-...-\bigl(A_{n}+k_{n}b_{n}\bigr)t_{n}}H\bigl(t_{1},...,t_{n}\bigr) \, dt_{1}... \, dt_{n}=0,
$$
for any $k=(k_{1},...,k_{n})\in {\mathbb N}_{0}^{n}.$ By (i), we get
$$
\int^{+\infty}_{0}\cdots  \int^{+\infty}_{0}\psi\bigl(t_{1},...,t_{n}\bigr)  H\bigl(t_{1},...,t_{n}\bigr) \, dt_{1} ... \, dt_{n}=0,
$$
for any $\psi \in L^{1}([0,+\infty)^{n}),$ which yields $H\equiv 0.$
\end{proof}

The following results do not seem to be  precisely formulated in the one-dimensional setting (cf. also \cite[Proposition 1.4]{mihai}): 

\begin{prop}\label{total}
Suppose that $\Re \lambda_{k}^{j} >0$ for all $k\in {\mathbb N}$ and $j\in {\mathbb N}_{n}$. Then the set
$\{\exp(-\lambda_{k}^{1}t_{1}-...-\lambda_{k}^{n}t_{n}) : k\in {\mathbb N}\}$ is total in $L^{1}([0,+\infty)^{n})$
if and only if the sequence $((\lambda_{k}^{1},...,\lambda_{k}^{n}))_{k\in {\mathbb N}}$ is a uniqueness sequence for the multidimensional Laplace transform.
\end{prop}

\begin{proof}
If the set
$\{\exp(-\lambda_{k}^{1}t_{1}-...-\lambda_{k}^{n}t_{n}) : k\in {\mathbb N}\}$ is total in $L^{1}([0,+\infty)^{n}),$
then we can repeat verbatim the argumentation contained in the proof of Proposition \ref{gade}(ii) to deduce that the sequence $((\lambda_{k}^{1},...,\lambda_{k}^{n}))_{k\in {\mathbb N}}$ is a uniqueness sequence for the multidimensional Laplace transform. Let us prove the converse statement. Suppose the contrary; then the Hanh-Banach theorem implies the existence of a non-trivial function $\psi \in L^{\infty}([0,+\infty)^{n})$ such that 
$$
\int^{+\infty}_{0}...\int^{+\infty}_{0}e^{-\lambda_{k}^{1}t_{1}-...-\lambda_{k}^{n}t_{n}}\psi\bigl(t_{1},...,t_{n}\bigr)\, dt_{1}\, ...\, dt_{n}=0,\quad k\in {\mathbb N}.
$$
In other words, $\hat{\psi}(\lambda_{k}^{1},...,\lambda_{k}^{n})=0,$ $k\in {\mathbb N}$ but $\psi \neq 0$ a.e., which contradicts our assumption that $((\lambda_{k}^{1},...,\lambda_{k}^{n}))_{k\in {\mathbb N}}$ is a uniqueness sequence for the multidimensional Laplace transform. 
\end{proof}

\begin{prop}\label{split}
Suppose that $m\in {\mathbb N},$ $((\lambda_{k}^{j,1},...,\lambda_{k}^{j,n}))_{k\in {\mathbb N}}$ is a sequence  in ${\mathbb C}^{n}$ for every $j\in {\mathbb N}_{m}$, there exist real numbers $\omega_{1},...,\omega_{n}$ such that $\omega_{1}<\Re \lambda_{k}^{j,1},...,\omega_{n}<\Re \lambda_{k}^{j,n}$ for every $k\in {\mathbb N}$ and $j\in {\mathbb N}_{m}$,
$((\lambda_{k}^{1},...,\lambda_{k}^{n}))_{k\in {\mathbb N}}$ is a sequence in ${\mathbb C}^{n}$ and, for every $k\in {\mathbb N},$ there exist $j\in {\mathbb N}_{m}$ and $s\in {\mathbb N}$ such that $(\lambda_{k}^{1},...,\lambda_{k}^{n})=(\lambda_{s}^{j,1},...,\lambda_{s}^{j,n}).$
If $((\lambda_{k}^{1},...,\lambda_{k}^{n}))_{k\in {\mathbb N}}$ is a uniqueness sequence for the multidimensional Laplace transform, then there exists $j\in {\mathbb N}_{m}$ such that $((\lambda_{k}^{j,1},...,\lambda_{k}^{j,n}))_{k\in {\mathbb N}}$ is a uniqueness sequence for the multidimensional Laplace transform.
\end{prop}

\begin{proof}
Suppose the contrary. Then, for every $j\in {\mathbb N}_{m},$ there exist non-trivial locally integrable function $f_{j} : [0,+\infty)^{n} \rightarrow {\mathbb C}$ and numbers $\omega_{j,1} \in [-\infty,+\infty),...,$ $ \omega_{j,n}\in [-\infty,+\infty)$ with $\Omega_{j}:=\{\lambda_{1}\in {\mathbb C} : \Re \lambda_{1}>\omega_{j,1}\} \times ... \times \{\lambda_{n}\in {\mathbb C} : \Re \lambda_{n}>\omega_{j,n}\} \subseteq \Omega_{b}(f_{j}) \cap \Omega(f_{j})$, $\omega_{j,s}<\Re \lambda_{k}^{j,s}$ for all $k\in {\mathbb N}$, $s\in {\mathbb N}_{n}$, and $\hat{f_{j}}(\lambda_{k}^{j,1},...,\lambda_{k}^{j,n})=0$ for all $k\in {\mathbb N}.$ Keeping in mind Proposition \ref{nmt}, the proof of Proposition \ref{gade}(ii) and the prescribed assumption that $\omega_{1}<\Re \lambda_{k}^{j,1},...,\omega_{n}<\Re \lambda_{k}^{j,n}$ for every $k\in {\mathbb N}$ and $j\in {\mathbb N}_{m}$, we may assume without loss of generality that $\omega_{j}>0$ and the function $f_{j}(\cdot)$ is continuous for all $j\in {\mathbb N}_{m}.$ Define
$f(t_{1},...,t_{n}):=(f_{1} \ast_{0} ... \ast_{0} f_{m})(t_{1},...t_{n}),$ $(t_{1},...t_{n})\in [0,+\infty)^{n}.$ Then Lemma \ref{nova} implies together with the assumption that, for every $k\in {\mathbb N},$ there exist $j\in {\mathbb N}_{m}$ and $s\in {\mathbb N}$ such that $(\lambda_{k}^{1},...,\lambda_{k}^{n})=(\lambda_{s}^{j,1},...,\lambda_{s}^{j,n}),$ that $\hat{f}(\lambda_{k}^{1},...,\lambda_{k}^{n})=\hat{f_{1}}(\lambda_{k}^{1},...,\lambda_{k}^{n})\cdot ... \cdot \hat{f_{m}}(\lambda_{k}^{1},...,\lambda_{k}^{n})=0,$ $k\in {\mathbb N},$ and therefore $f(t_{1},...,t_{n})=0,$ $t\in [0,+\infty)^{n}.$ This easily implies $f_{j}=0$ a.e. ($j\in {\mathbb N}_{m}$), which is a contradiction.
\end{proof}              

In particular, if $m\in {\mathbb N}$,
$((\lambda_{k}^{1},...,\lambda_{k}^{n}))_{k\in {\mathbb N}}$ is a uniqueness sequence for the multidimensional Laplace transform and there exist real numbers $\omega_{1},...,\omega_{n}$ such that $\omega_{1}<\Re \lambda_{k}^{1},...,\omega_{n}<\Re \lambda_{k}^{n}$ for every $k\in {\mathbb N}$, then Proposition \ref{split} yields that at least one of the sequences 
\begin{align*} \bigl
((\lambda_{km}^{1},...,\lambda_{km}^{n})\bigr)_{k\in {\mathbb N}},\ \bigl((\lambda_{km+1}^{1},...,\lambda_{km+1}^{n})\bigr)_{k\in {\mathbb N}}, ...,\ \bigl((\lambda_{km+(m-1)}^{1},...,\lambda_{km+(m-1)}^{n})\bigr)_{k\in {\mathbb N}},
\end{align*} 
is a uniqueness sequence for the multidimensional Laplace transform.                             

It is worth observing that Proposition \ref{gade}(ii) can be deduced directly from \cite[Proposition 1.7.2]{a43} and the following result:

\begin{thm}\label{tor}
Suppose that $\Re \lambda_{k}^{j} >0$ for all $k\in {\mathbb N}$ and $1\leq j\leq n.$ Then $(\lambda_{k}^{1})_{k\in {\mathbb N}},...,\ (\lambda_{k}^{n})_{k\in {\mathbb N}}$ are the uniqueness sequences for the one-dimensional Laplace transform if and only if $((\lambda_{k_{1}}^{1},...,\lambda_{k_{n}}^{n}))_{  (k_{1},...,k_{n})\in {\mathbb N}^{n}}$ is a uniqueness sequence for the multidimensional Laplace transform.
\end{thm}

\begin{proof}
Let $(\lambda_{k}^{1})_{k\in {\mathbb N}},...,\ (\lambda_{k}^{n})_{k\in {\mathbb N}}$ be the uniqueness sequences for  one-dimensional Laplace transform, and let a locally integrable function $f : [0,+\infty)^{n} \rightarrow X$ satisfy that there exist numbers $\omega_{1} \in [-\infty,+\infty),...,\ \omega_{n}\in [-\infty,+\infty)$ such that $\Omega \subseteq \Omega_{b}(f)\cap \Omega(f)$, $\omega_{j}<\Re \lambda_{k}^{j}$ for all $k\in {\mathbb N}$, $j\in {\mathbb N}_{n}$, and $\hat{f}(\lambda_{k_{1}}^{1},...,\lambda_{k_{n}}^{n})=0$ for all $k\in {\mathbb N}$. Define the function $G(\cdot)$ in the same way as in the proof of Proposition \ref{gade}(ii). Taking into account \cite[Theorem 2.10(ii)-(iii)]{mvlt} and the assumption $\Re \lambda_{k}^{j} >0$ for all $k\in {\mathbb N}$ and $1\leq j\leq n$, we get 
$\hat{G}(\lambda_{k_{1}}^{1},...,\lambda_{k_{n}}^{n})=0$ for all $(k_{1},...,k_{n})\in {\mathbb N}^{n};$ furthermore, fixing numbers $t_{2}\geq 0,...,\ t_{n}\geq 0$ and tuple $(k_{2},...,k_{n})\in {\mathbb N}^{n-1},$ we can apply the 
Fubini theorem in order to see that
$$
\int^{+\infty}_{0}e^{-\lambda_{k_{1}}^{1}t_{1}}\Biggl[ \int^{+\infty}_{0}\cdots \int^{+\infty}_{0} e^{-\lambda_{k_{2}}^{2}t_{2}-...-\lambda_{k_{n}}^{n}t_{n}}G\bigl( t_{1},t_{2},...,t_{n}\bigr)\, dt_{2} ....\, dt_{n}\Biggr]\, dt_{1}=0,
$$
for any $k_{1}\in {\mathbb N}.$ Hence,
$$
 \int^{+\infty}_{0}\cdots \int^{+\infty}_{0} e^{-\lambda_{k_{2}}^{2}t_{2}-...-\lambda_{k_{n}}^{n}t_{n}}G\bigl( t_{1},t_{2},...,t_{n}\bigr)\, dt_{2} ....\, dt_{n}=0,
$$
for any $t_{1}\geq 0$ and $(k_{2},...,k_{n})\in {\mathbb N}^{n-1}.$ Repeating this procedure, we get
$G\equiv 0,$ which clearly 
implies $f=0$ a.e. and completes the first part of proof. 

Conversely, suppose that $((\lambda_{k_{1}}^{1},...,\lambda_{k_{n}}^{n}))_{  (k_{1},...,k_{n})\in {\mathbb N}^{n}}$ is a uniqueness sequence for the multidimensional Laplace transform. We will only prove that $(\lambda_{k}^{1})_{k\in {\mathbb N}}$ is a uniqueness sequence for the one-dimensional Laplace transform.
Towards this end, suppose that $f\in L_{loc}^{1}([0,+\infty)),$ abs$(f)<\Re \lambda_{k}^{1}$, $k\in {\mathbb N}$ and $\hat{f}(\lambda_{k}^{1})=0$ for all $k\in {\mathbb N}.$ Let us consider the continuous function $t\mapsto f^{[1]}(t):=\int^{t}_{0}f(s)\, ds,$ $t\geq 0.$ Define $f_{1}(t_{1},...,t_{n}):=f^{[1]}(t_{1}),$ $(t_{1},...,t_{n}) \in [0,+\infty)^{n},$
$\omega_{1}:=\max(\text{abs}(f),0)$ and $\omega_{2}:=...:=\omega_{n}:=0.$ Then we have $\omega_{1} \in [-\infty,+\infty)$, $\Omega \subseteq \Omega_{b}(f_{1})\cap \Omega(f_{1})$, $\omega_{j}<\Re \lambda_{k}^{j}$ for all $k\in {\mathbb N}$,  $j\in {\mathbb N}_{n}$, and $\hat{f_{1}}(\lambda_{k_{1}}^{1},...,\lambda_{k_{n}}^{n})=\hat{f}(\lambda_{k_{1}}^{1})/(\lambda_{k_{1}}^{1}\cdot ... \cdot \lambda_{k_{n}}^{n})=0$ for all $(k_{1},...,k_{n})\in {\mathbb N}^{n}$. Consequently, we have $f^{[1]}(t_{1})=0$, $t_{1}\geq 0$ and $f=0$ a.e., which completes the second part of proof.
\end{proof} 

The subsequent result can be deduced following the lines of proof of Theorem \ref{tor}, by considering the function{\small
$$
f_{j_{1},...,j_{s}}\bigl(t_{1},...,t_{n}\bigr):=\int^{t_{j_{1}}}_{0}\cdots \int^{t_{j_{s}}}_{0} f\bigl(r_{j_{1}},...,r_{j_{s}}\bigr)\, dr_{j_{1}}...\, dr_{j_{s}},\ \bigl(t_{1},...,t_{n}\bigr) \in [0,+\infty)^{n} 
$$}in place of function $f_{1}(\cdot)$ therein:

\begin{prop}\label{zxc}
Suppose that $\Re \lambda_{k}^{j} >0$ for all $k\in {\mathbb N}$ and $1\leq j\leq n,$ $1\leq s\leq n$ and $1\leq j_{1}<...<j_{s}\leq n$. If the sequence $( (\lambda_{k}^{1},...,\lambda_{k}^{n}))_{ k\in {\mathbb N}}$ is a uniqueness sequence for the multidimensional Laplace transform, then $((\lambda_{k}^{j_{1}},...,\lambda_{k}^{j_{s}}))_{k\in {\mathbb N}}$ is a uniqueness sequence for the multidimensional Laplace transform in ${\mathbb R}^{s}.$
\end{prop}

\begin{rem}\label{pot}
If $P : {\mathbb N}^{n} \rightarrow {\mathbb N}$ is a bijection, then it is clear that
$( \lambda_{k}=(\lambda_{k}^{1},...,\lambda_{k}^{n}))_{ k\in {\mathbb N}}$ is a uniqueness sequence for the multidimensional Laplace transform if and only if $((\lambda_{P(k_{1},...,k_{n})}^{1},...,\lambda_{P(k_{1},...,k_{n})}^{n}))_{  (k_{1},...,k_{n})\in {\mathbb N}^{n}}$ is a uniqueness sequence for the multidimensional Laplace transform.
\end{rem}

We feel it is our duty to emphasize that Theorem \ref{tor} and Proposition \ref{zxc} can be difficult for applications because the sequences $((\lambda_{k_{1}}^{1},...,\lambda_{k_{n}}^{n}))_{  (k_{1},...,k_{n})\in {\mathbb N}^{n}}$ and $((\lambda_{k}^{j_{1}},...,\lambda_{k}^{j_{s}}))_{k\in {\mathbb N}}$ form the rectangular lattices in ${\mathbb R}^{n}$  and ${\mathbb R}^{s},$ respectively. Some illustrative applications are given below (cf. also \cite{a43}):

\begin{example}\label{or}
\begin{itemize}
\item[(i)] If $a_{i}>0,$ $b_{i}>0$ and $0<\gamma_{i} \leq 1$ for $1\leq i\leq n,$ then $((a_{1}+k_{1}^{\gamma_{1}}b_{1},...,a_{n}+k_{n}^{\gamma_{n}}b_{n}))_{(k_{1},...,k_{n})\in {\mathbb N}^{n}}$ is a uniqueness sequence for the multidimensional Laplace transform. If there exists $i\in {\mathbb N}_{n}$ such that $\gamma_{i}>1,$ then the above sequence is not a uniqueness sequence for the multidimensional Laplace transform.
\item[(ii)] If $0<\gamma_{i} \leq 1/2$ for $1\leq i\leq n,$ then $((1+ik_{1}^{\gamma_{1}},...,1+ik_{n}^{\gamma_{n}}))_{(k_{1},...,k_{n})\in {\mathbb N}^{n}}$ is a uniqueness sequence for the multidimensional Laplace transform. If there exists $i\in {\mathbb N}_{n}$ such that $\gamma_{i}>1/2,$ then the above sequence is not a uniqueness sequence for the multidimensional Laplace transform.
\item[(iii)] Suppose that, for every $j\in {\mathbb N}_{n}$, $(\lambda_{k}^{j})_{k\in {\mathbb N}}$ is an infinite sequence without accumulation points such that $\Re \lambda_{k}^{j}\geq \delta>0,$ $k\in {\mathbb N}$ for some $\delta>0$, and
\begin{align}\label{retz}
\sum_{k=1}^{+\infty}\Biggl[ 1-\Biggl| \frac{  \lambda_{k}^{j}-1 }{  \lambda_{k}^{j}+1 }\Biggr|\Biggr]=+\infty.
\end{align}
Due to Theorem \ref{ibeer} and Theorem \ref{tor}, $((\lambda_{k_{1}}^{1},...,\lambda_{k_{n}}^{n}))_{  (k_{1},...,k_{n})\in {\mathbb N}^{n}}$ is a uniqueness sequence for the multidimensional Laplace transform. Conversely, if there exists an index  $j\in {\mathbb N}_{n}$ such that $(\lambda_{k}^{j})_{k\in {\mathbb N}}$ is a complex sequence without accumulation points, $\Re \lambda_{k}^{j}\geq \delta>0,$ $k\in {\mathbb N}$ for some $\delta>0$, the sum in \eqref{retz} is finite and $|\arg(\lambda_{k}^{j})|\leq \theta <\pi/2$, $k\in {\mathbb N}$, then Theorem \ref{ibeer} and and Theorem \ref{tor} imply that $((\lambda_{k_{1}}^{1},...,\lambda_{k_{n}}^{n}))_{  (k_{1},...,k_{n})\in {\mathbb N}^{n}}$ is not a uniqueness sequence for the multidimensional Laplace transform.
\end{itemize}
\end{example}

If $(\lambda_{k}^{1})_{k\in {\mathbb N}},...,\ (\lambda_{k}^{n})_{k\in {\mathbb N}}$ are uniqueness sequences for one-dimensional Laplace transform, then $((\lambda_{k}^{1},...,\lambda_{k}^{n}))_{  k\in {\mathbb N}}$ need not be uniqueness sequence for multidimensional Laplace transform (cf. Theorem \ref{tor}):

\begin{example}\label{oro}
The sequences $(\lambda_{k}^{1})_{k\in {\mathbb N}}$ and $(\lambda_{k}^{2})_{k\in {\mathbb N}}$, where $\lambda_{k}^{1}:=\lambda_{k}^{2}:=k,$ $k\in {\mathbb N}$, are the uniqueness sequences for one-dimensional Laplace transform. But, the sequence $((k,k))_{k\in {\mathbb N}}$ is not a uniqueness sequence for the two-dimensional Laplace transform since the function $f: [0,+\infty)^{2} \rightarrow {\mathbb R},$ given by $f(t_{1},t_{2}):=g_{2}(t_{1}) \cdot g_{3}(t_{2})-g_{3}(t_{1}) \cdot g_{2}(t_{2})$, $t_{1}\geq 0,$ $t_{2}\geq 0$ satisfies $F(\lambda_{1},\lambda_{2})=(\lambda_{1}-\lambda_{2})/(\lambda_{1}^{3}\cdot \lambda_{2}^{3}),$ $\Re \lambda_{1}>0,$ $\Re \lambda_{2}>0$ and therefore $F(k,k)=0$ for all $k\in {\mathbb N}.$ 

Furthermore, if $m,\ s\in {\mathbb N}$, $c_{j}\in {\mathbb N}$ for $1\leq  j\leq m$ and $d_{j}\in {\mathbb N}$ for $1\leq  j\leq s$, then the inverse Laplace transform $f(\cdot,\cdot)$ of function
$$     
F\bigl(\lambda_{1},\lambda_{2}\bigr):=\frac{\lambda_{1}-1}{\lambda_{1}^{3}\cdot \lambda_{2}^{3}}\cdot \frac{\lambda_{2}-1}{\lambda_{1}^{3}\cdot \lambda_{2}^{3}}\cdot \prod_{j=1}^{m}\frac{\lambda_{1}-c_{j}\lambda_{2}}{\lambda_{1}^{3}\cdot \lambda_{2}^{3}}\cdot \prod_{j=1}^{s}\frac{d_{j}\lambda_{1}-\lambda_{2}}{\lambda_{1}^{3}\cdot \lambda_{2}^{3}},\quad \Re \lambda_{1}>0,\ \Re \lambda_{2}>0
$$
is continuous and non-trivial, so that the sequence 
$$
\{(k,1) : k\in {\mathbb N}\} \cup \{(1,k) : k\in {\mathbb N}\} \bigcup_{j=1}^{m}\Bigl\{ \bigl(c_{j}k,k\bigr) : k\in {\mathbb N}\Bigr\} \bigcup_{j=1}^{s}\Bigl\{ \bigl(k,d_{j}k\bigr) : k\in {\mathbb N}\Bigr\}
$$ 
is not a uniqueness sequence for the double Laplace transform.

It is logical to ask whether the sequence $\{ (k_{1},k_{2}) : (k_{1},k_{2})\in {\mathbb N}^{2},\ \arg(k_{1}+ik_{2}) \leq \theta \}$, where $0<\theta <\pi/2,$ is a uniqueness sequence for the double Laplace transform. The answer is affirmative and we will first prove this statement in case $\pi/4 \leq \theta <\pi/2.$ Suppose that a locally integrable function $f : [0,+\infty)^{2} \rightarrow {\mathbb C}$  
satisfies that there exist real numbers $\omega_{1}<1$ and $\omega_{2}<1$ such that $\{ \lambda_{1} \in {\mathbb C} : \Re \lambda_{1}>\omega_{1}\} \times \{\lambda_{2} \in {\mathbb C} : \Re \lambda_{2}>\omega_{2} \}\subseteq \Omega(f) \cap \Omega_{b}(f),$ and
$$
\int^{+\infty}_{0}\int^{+\infty}_{0}e^{-k_{1}t_{1}-k_{2}t_{2}}f\bigl( t_{1},t_{2}\bigr)\, dt_{1}\, dt_{2}=0\quad \mbox{ for all } \bigl(k_{1},k_{2}\bigr)\in {\mathbb N}^{2}\mbox{ with }k_{1}\geq k_{2}.
$$
Define 
$$
f^{[1]}(t):=\int^{t_{1}}_{0}\int^{t_{2}}_{0}f\bigl( s_{1},s_{2}\bigr)\, ds_{1}\, ds_{2}\mbox{ and }g\bigl(t_{1},t_{2}\bigr):=f^{[1]}(t_{2},t_{1}),
$$ 
for any $(t_{1},t_{2})\in  [0,+\infty)^{2}.$
Due to \cite[Theorem 2.10(ii)-(iii)]{mvlt}, there exists a real number $M\geq 1$ such that $|f^{[1]}(t_{1},t_{2})|+|g(t_{1},t_{2})| \leq Me^{t_{1}+t_{2}}$, $(t_{1},t_{2})\in  [0,+\infty)^{2},$ and
$$
\int^{+\infty}_{0}\int^{+\infty}_{0}e^{-k_{2}t_{1}-k_{1}t_{2}}g\bigl( t_{1},t_{2}\bigr)\, dt_{1}\, dt_{2}=0\quad \mbox{ for all } \bigl(k_{1},k_{2}\bigr)\in {\mathbb N}^{2}\mbox{ with }k_{1}\geq k_{2}.
$$ 
Set now $h:= f^{[1]} \ast_{0}g. $ Then there exists a real number $M'\geq 1$ such that $|h(t_{1},t_{2})| \leq M'e^{t_{1}+t_{2}}$, $(t_{1},t_{2})\in  [0,+\infty)^{2}$ and $\hat{h}(k_{1},k_{2})=0$ for all $(k_{1},k_{2})\in {\mathbb N}^{2}.$ This simply implies $f^{[1]}\equiv g\equiv 0$ and therefore $f=0$ a.e., as required.

Now we will solve the problem for an angle $0<\theta <\pi/4.$ Define the sequence $(\theta_{k})_{k\in {\mathbb N}}$ recursively by $\theta_{1}:=\arctan (1/2)$ and $\theta_{k+1}:=\arctan (\theta_{k}/2),$ $k\in {\mathbb N}.$ Then $0<\theta_{k}<2^{-k},$ $k\in {\mathbb N}$ and we only need to prove that, for every $k\in {\mathbb N},$ $\{ (k_{1},k_{2}) : (k_{1},k_{2})\in {\mathbb N}^{2},\ \arg(k_{1}+ik_{2}) \leq \theta_{k} \}$ is a uniqueness sequence for the double Laplace transform. The main problem is to prove the validity of the above statement for the angle $\theta_{1}=\arctan (1/2)$, since we can inductively repeat the same procedure infinitely many times. Fist of all, let $c_{1}>0,$ $c_{2}>0$ and let us consider the function $f_{c_{1},c_{2}} : [0,+\infty)^{2} \rightarrow {\mathbb C}$ defined by $f_{c_{1},c_{2}}(t_{1},t_{2}):=f(c_{1}t_{1},c_{2}t_{2}),$ $(t_{1},t_{2})\in [0,+\infty)^{2}.$ Then $\widehat{f_{c_{1},c_{2}}}(\lambda_{1},\lambda_{2})=\hat{f}(\lambda_{1}/c_{1},\lambda_{2}/c_{2})$ and the assumption $\hat{f}(k_{1},k_{2})=0$ for all $(k_{1},k_{2})\in {\mathbb N}^{2}$ with $\arg(k_{1}+ik_{2}) \leq  \arctan (1/2)$ implies that there exists a continuous Laplace transformable function $f_{0} : [0,+\infty)^{2} \rightarrow {\mathbb C}$ such that the Laplace transform of function $f_{0}\ast_{0} f \ast_{0}f_{1,2}(\cdot,\cdot)$ annulates all points $(k_{1},k_{2})\in {\mathbb N}^{2}$ with $\arg(k_{1}+ik_{2}) \leq  \arctan (1/2)$, as well as all points $(k_{1},k_{2})\in {\mathbb N}^{2}$ such that $\arctan (1/2)<\arg(k_{1}+ik_{2}) \leq  \pi/4$ and $k_{2}$ is an even number. Arguing as in the corresponding part of proof with $\theta =\pi/4,$ it follows that there exists a continuous Laplace transformable function $g : [0,+\infty)^{2} \rightarrow {\mathbb C}$ such that the Laplace transform of function $f_{0}\ast_{0} f \ast_{0}f_{1,2} \ast_{0} g (\cdot,\cdot)$ annulates all points $(k_{1},k_{2})\in {\mathbb N}^{2}$ with $\arg(k_{1}+ik_{2}) \leq  \arctan (1/2)$, all points $(k_{1},k_{2})\in {\mathbb N}^{2}$ such that $\arctan (1/2)<\arg(k_{1}+ik_{2}) <  (\pi/4)+\arctan (1/2)$ and $k_{2}$ is an even number, as well as all points 
$(k_{1},k_{2})\in {\mathbb N}^{2}$ such that $(\pi/4)+\arctan (1/2) \leq \arg(k_{1}+ik_{2}) <  (\pi/2).$ Now it can be easily explained that 
the Laplace transform of function $[(f_{0}\ast_{0} f \ast_{0}f_{1,2} \ast_{0} g) \ast (f_{0}\ast_{0} f \ast_{0}f_{1,2} \ast_{0} g) _{1,1/2}](\cdot,\cdot)$ annulates all points $(k_{1},k_{2})\in {\mathbb N}^{2}$, which simply implies that $f=0$ a.e., as required.
\end{example}

By Theorem \ref{tor} and Proposition \ref{zxc}, any
uniqueness sequence $((\lambda_{k}^{1},...,\lambda_{k}^{n}))_{k\in {\mathbb N}}$ for the multidimensional Laplace transform which satisfies $\Re \lambda_{k}^{j} >0$ for all $k\in {\mathbb N}$ and $j\in {\mathbb N}_{n}$ 
is a subsequence of the uniqueness sequence for the multidimensional Laplace transform $((\lambda_{k_{1}}^{1},...,\lambda_{k_{n}}^{n}))_{  (k_{1},...,k_{n})\in {\mathbb N}^{n}}$, where $(\lambda_{k}^{1})_{k\in {\mathbb N}},...,$ and $ (\lambda_{k}^{n})_{k\in {\mathbb N}}$ are the uniqueness sequences for Laplace transform.

Suppose now that $\lambda_{k}^{j}\geq 1$ for all $k\in {\mathbb N}$ and $j\in {\mathbb N}_{n}.$ Since the mapping $\Phi : L^{1}((0,1)^{n}) \rightarrow L^{1}((0,+\infty)^{n})$, given by
$
[\Phi (g)](t_{1},...,t_{n}):=\exp(-t_{1}-...-t_{n})g( \exp(-t_{1}),...,\exp(-t_{n})),$ $(t_{1},...,t_{n}) \in [0,+\infty)^{n} ,
$
$ g\in L^{1}((0,1)^{n}),$ is an isometric isomorphism which maps 
$$ 
t_{1}^{ \lambda_{k_{1}}^{1}-1}\cdot ... \cdot t_{n}^{ \lambda_{k_{n}}^{n}-1}\mbox{ to }e^{-\lambda_{k_{1}}^{1}t_{1} -...-\lambda_{k_{n}}^{n}t_{n}}\mbox{ for all }(k_{1},...,k_{n})\in  {\mathbb N}^{n},
$$
Proposition \ref{total} implies that $((\lambda_{k_{1}}^{1},...,\lambda_{k_{n}}^{n}))_{  (k_{1},...,k_{n})\in S},$ where $\emptyset \neq S\subseteq {\mathbb N}^{n},$ will be a uniqueness sequence for multidimensional Laplace transform if the set 
\begin{align}\label{wws}
{\mathcal P}:=\Biggl\{t_{1}^{ \lambda_{k_{1}}^{1}-1}\cdot ... \cdot t_{n}^{ \lambda_{k_{n}}^{n}-1} : (k_{1},...,k_{n})\in  S\Biggr\}
\end{align}
is total in $C([0,1]^{n}).$ The Stone-Weierstrass theorem yields that the set $
{\mathcal P}$ is total in $C([0,1]^{n})$ if the following three conditions hold:
\begin{itemize}
\item[(i)] There exists $ (k_{1},...,k_{n})\in  S$ such that $\lambda_{k_{1}}^{1}=...=\lambda_{k_{n}}^{n}=1.$
\item[(ii)] For every $j\in {\mathbb N}_{n}$, there exists $ (k_{1},...,k_{n})\in  S$ such that $\lambda_{k_{j}}^{j}\neq 1$ and $\lambda_{k_{s}}^{s}=1$ for all $s\in {\mathbb N}_{n}\setminus \{j\}.$
\item[(iii)] For every $ (k_{1},...,k_{n})\in  S$ and $ (k_{1}',...,k_{n}')\in  S,$ there exists $ (k_{1}'',...,k_{n}'')\in  S$ such that 
$$
\Bigl(\lambda_{k_{1}}^{1},...,\lambda_{k_{n}}^{n}\Bigr)+\Bigl(\lambda_{k_{1'}}^{1},...,\lambda_{k_{n'}}^{n}\Bigr)=\Bigl(\lambda_{k_{1}''}^{1},...,\lambda_{k_{n}''}^{n}\Bigr)+\Bigl(1,...,1\Bigr).
$$
\end{itemize}
We need the last condition since the linear span of ${\mathcal P}$ has to be closed under pointwise multiplication of functions. But, then a relatively simple analysis shows that the sequence $((\lambda_{k_{1}}^{1},...,\lambda_{k_{n}}^{n}))_{  (k_{1},...,k_{n})\in S} $ must contain a subsequence of the form $((a_{1}+k_{1}b_{1},...,a_{n}+k_{n}b_{n}))_{ (k_{1},...,k_{n})\in {\mathbb N}_{0}^{n}}$, for certain real numbers $a_{i}>$ and $b_{i}>0$ for $1\leq i\leq n$, which is quite unsatisfactory (cf. Proposition \ref{gade}).

On the other hand, we can employ the multivariate M\"untz theorem mentioned in the introductory part of paper to obtain a sufficient condition under which the set $
{\mathcal P}$ is total in $C([0,1]^{n}).$ Now we will formulate and prove the following partial solution to Problem 1 (cf. also \cite[Satz $5^{32}$]{dech}):

\begin{thm}\label{kunja}
Suppose that $\emptyset \neq S\subseteq {\mathbb N}^{n} $ and $\Re \lambda_{k}^{j}\geq 1$ for all $k\in {\mathbb N}$ and $j\in {\mathbb N}_{n}.$
Then $((\lambda_{k_{1}}^{1},...,\lambda_{k_{n}}^{n}))_{  (k_{1},...,k_{n})\in S} $ is a uniqueness sequence for multidimensional Laplace transform if the following conditions hold:
\begin{itemize}
\item[(i)]  $\inf\bigl\{ |\Re\lambda_{k_{1}}^{1}-\Re\lambda_{k_{1}'}^{1}|+...|\Re\lambda_{k_{n}}^{n}-\Re\lambda_{k_{n}'}^{n}|:  (k_{1},...,k_{n})\in  S,\ (k_{1}',...,k_{n}')\in  S \bigr\} >0.$
\item[(ii)] $\limsup_{t\rightarrow +\infty}\frac{\text{card}\bigl(\bigl\{(k_{1},...,k_{n})\in S: |\Re\lambda_{k_{1}}^{1}-1| +...+|\Re\lambda_{k_{n}}^{n}-1| \leq t\bigr\}\bigr)}{t^{n}}>0.$
\item[(iii)] For every $j\in {\mathbb N}_{n},$ there exists a real constant $c_{j}\in {\mathbb R}$ such that $\Im \lambda_{k_{j}}^{j}=c_{j}$ for all $(k_{1},...,k_{n})\in  S.$  
\end{itemize}
\end{thm}

\begin{proof}
Suppose that a locally integrable function $f : [0,+\infty)^{n}\rightarrow {\mathbb C}$ satisfies
$$
\int^{+\infty}_{0}...\int^{+\infty}_{0}e^{-(\Re \lambda_{k_{1}}^{1}+i \Im \lambda_{k_{1}}^{1})t_{1}-...-(\Re \lambda_{k_{n}}^{n}+i \Im \lambda_{k_{n}}^{n})t_{n}}f\bigl(t_{1},...,t_{n}\bigr)\, dt_{1}\, ...\, dt_{n}=0,
$$
for any $(k_{1},...,k_{n})\in  S.$ Then the assumption (iii) yields
$$
\int^{+\infty}_{0}...\int^{+\infty}_{0}e^{-\Re \lambda_{k_{1}}^{1}t_{1}-...-\Re \lambda_{k_{n}}^{n}t_{n}}\Bigl[ e^{-ic_{1}t_{1}-...-ic_{n}t_{n}} f\bigl(t_{1},...,t_{n}\bigr)\Bigr]\, dt_{1}\, ...\, dt_{n}=0,
$$
for any $(k_{1},...,k_{n})\in  S.$ Keeping in mind (i)-(ii), we can apply the above-mentioned multivariate  M\"untz type theorem to obtain that the set $
{\mathcal P}_{1},$ obtained by replacing the numbers $\lambda_{k_{1}}^{1},....,\ \lambda_{k_{n}}^{n}$with the numbers $\Re \lambda_{k_{1}}^{1},....,\ \Re \lambda_{k_{n}}^{n}$ in \eqref{wws}, respectively,
is total in $C([0,1]^{n})$. Now Proposition \ref{total} implies that  $e^{-ic_{1}t_{1}-...-ic_{n}t_{n}} f(t_{1},...,t_{n})=0$ for a.e. $(t_{1},...,t_{n})\in [0,+\infty)^{n},$ which simply completes the proof.
\end{proof}

The notion of an $n$-dimensional M\"untz sequence was introduced in \cite[Definition 1]{kroo}. Unfortunately, we cannot employ here the multivariate M\"untz theorem established in \cite[Theorem A]{kroo} since it is completely devoted to the study of totality of set $
{\mathcal P}$
 in $C(I),$ where $I $ is a strictly positive domain, i.e., $I\subseteq (0,+\infty)^{n}.$

\subsection{Subordination principles for uniqueness sequences}\label{subasub}

We open this subsection by stating the following result (in the proof, we use the Wright function $\Phi_{\gamma}(\cdot),$ where $0<\gamma<1;$ cf. \cite{FKP} for the notion): 

\begin{thm}\label{ressuba}
Suppose that $1\leq l\leq n,$ $1\leq j\leq {n \choose l},$ $D_{l,j}=\{j_{1},...,j_{l}\}$ is a fixed subset of ${\mathbb N}_{n},$ where $1\leq j_{1}<...<j_{l}\leq n,$ $\gamma_{j_{i}} \in (0,1)$ for all $i\in {\mathbb N}_{l},$ and $\gamma_{l,j}=(\gamma_{j_{1}},...,\gamma_{j_{l}}).$ Suppose that $((\lambda_{k}^{1},...,\lambda_{k}^{n}))_{k\in {\mathbb N}}$ is a uniqueness sequence for the multidimensional Laplace transform, 
$\Re \lambda_{k}^{j} >0$ for all $k\in {\mathbb N}$ and $1\leq j\leq n,$ and $
\lim_{k\rightarrow +\infty}\Re \lambda_{k}^{j_{s}}=0,$ $ 1\leq s\leq l.$
Then the sequence{\small
\begin{align*}
\Biggl(\lambda_{k}^{1},...,\lambda_{k}^{j_{1}-1},\bigl(\lambda_{k}^{j_{1}}\bigr)^{\gamma_{j_{1}}},\lambda_{k}^{j_{1}+1},...,\lambda_{k}^{j_{2}-1},\bigl(\lambda_{k}^{j_{2}}\bigr)^{\gamma_{j_{2}}},\lambda_{k}^{j_{2}+1},...,\lambda_{k}^{j_{l}-1}, \bigl(\lambda_{k}^{j_{l}}\bigr)^{\gamma_{j_{l}}},\lambda_{k}^{j_{l}+1},...,\lambda_{k}^{n} \Biggr)_{k\in {\mathbb N}}
\end{align*}}
 is a uniqueness sequence for the multidimensional Laplace transform.
\end{thm}

\begin{proof}
Suppose that $f\in L_{loc}^{1}([0,+\infty)^{n}),$ there exist numbers $\omega_{1}' \in [-\infty,+\infty),...,$ $ \omega_{n}'\in [-\infty,+\infty)$ such that $\{\lambda_{1}\in {\mathbb C} : \Re \lambda_{1}>\omega_{1}'\} \times ... \times \{\lambda_{n}\in {\mathbb C} : \Re \lambda_{n}>\omega_{n}'\} \subseteq \Omega_{b}(f) \cap \Omega(f),$ $\omega_{j}'<\Re \lambda_{k}^{j}$ for all $k\in {\mathbb N},$ $j\in {\mathbb N}_{n} \setminus D_{l,j}$, $\omega_{j}'<\Re [(\lambda_{k}^{j})^{\gamma_{j}}]$ for all $k\in {\mathbb N}$, $j\in D_{l,j}$ and {\scriptsize
\begin{align*} \hat{f}
\Biggl(\lambda_{k}^{1},...,\lambda_{k}^{j_{1}-1},\bigl(\lambda_{k}^{j_{1}}\bigr)^{\gamma_{j_{1}}},\lambda_{k}^{j_{1}+1},...,\lambda_{k}^{j_{2}-1},\bigl(\lambda_{k}^{j_{2}}\bigr)^{\gamma_{j_{2}}},\lambda_{k}^{j_{2}+1},...,\lambda_{k}^{j_{l}-1}, \bigl(\lambda_{k}^{j_{l}}\bigr)^{\gamma_{j_{l}}},\lambda_{k}^{j_{l}+1},...,\lambda_{k}^{n} \Biggr)=0,
\end{align*}}for all $k\in {\mathbb N}.$  Define $G(\cdot)$ through \eqref{kw}.  
By \cite[Theorem 2.10(ii)-(iii)]{mvlt}, there exists a real number $M\geq 1$ such that $|G(t_{1},...,t_{n})| \leq M\exp(\omega_{1}t_{1}+...+\omega_{n}t_{n}),$ $ (t_{1},...,t_{n} ) \in [0,+\infty)^{n},$ where $\omega_{j}:=\max(\omega_{j}',0)$ for $1\leq j\leq n,$ and{\scriptsize
\begin{align}\label{profs} \hat{G}
\Biggl(\lambda_{k}^{1},...,\lambda_{k}^{j_{1}-1},\bigl(\lambda_{k}^{j_{1}}\bigr)^{\gamma_{j_{1}}},\lambda_{k}^{j_{1}+1},...,\lambda_{k}^{j_{2}-1},\bigl(\lambda_{k}^{j_{2}}\bigr)^{\gamma_{j_{2}}},\lambda_{k}^{j_{2}+1},...,\lambda_{k}^{j_{l}-1}, \bigl(\lambda_{k}^{j_{l}}\bigr)^{\gamma_{j_{l}}},\lambda_{k}^{j_{l}+1},...,\lambda_{k}^{n} \Biggr)=0,
\end{align}}for all $k\in {\mathbb N}.$
Set now{\footnotesize
\begin{align}\notag &
G_{\gamma_{l,j}}\bigl(t_{1},...,t_{n}\bigr):=\int^{+\infty}_{0}...\int^{+\infty}_{0} \Phi_{\gamma_{j_{1}}}\bigl(s_{j_{1}}\bigr)\cdot ... \cdot \Phi_{\gamma_{j_{l}}}\bigl(s_{j_{l}}\bigr) 
\\& \label{ubicaprim} \times G\Bigl(t_{1},...,t_{j_{1}-1},s_{j_{1}}t_{j_{1}}^{\gamma_{j_{1}}},t_{j_{1}+1},...,t_{j_{2}-1},s_{j_{2}}t_{j_{2}}^{\gamma_{j_{2}}},t_{j_{2}+1},...,t_{j_{l}-1}, s_{j_{l}}t_{j_{l}}^{\gamma_{j_{l}}},t_{j_{l}+1},...,t_{n}\Bigr)\, ds_{j_{1}}...\, ds_{j_{l}},
\end{align}}for any $t=(t_{1},...,t_{n})\in [0,+\infty)^{n}.$ By \cite[Theorem 2.2]{new}, we know the following: If $\lambda_{i} \in {\mathbb C}$ and $\Re (\lambda_{i})>\omega_{i}^{1/\gamma_{i}}$ for all $i\in D_{l,j}$ and $\Re (\lambda_{i})>\omega_{i}$ for all $i\in {\mathbb N}_{n} \setminus D_{l,j}$, then $\lambda=(\lambda_{1},...,\lambda_{n}) \in \Omega_{abs}(G_{\gamma_{l,j}})$, and {\footnotesize
\begin{align} \notag &
 \widehat {G_{\gamma_{l,j}}}\bigl(\lambda_{1},...,\lambda_{n}\bigr)=\lambda_{j_{1}}^{\gamma_{j_{1}}-1}\cdot ... \cdot \lambda_{j_{l}}^{\gamma_{j_{l}}-1}
\\\label{akjot}& \times \hat{G}\Bigl(\lambda_{1},...,\lambda_{j_{1}-1},\lambda_{j_{1}}^{\gamma_{j_{1}}},\lambda_{j_{1}+1},...,\lambda_{j_{2}-1},\lambda_{j_{2}}^{\gamma_{j_{2}}},\lambda_{j_{2}+1},...,\lambda_{j_{l}-1}, \lambda_{j_{l}}^{\gamma_{j_{l}}},\lambda_{j_{l}+1},...,\lambda_{n} \Bigr).
\end{align}}
Since we have assumed that $
\lim_{k\rightarrow +\infty}\Re \lambda_{k}^{j_{s}}=0,$ $ 1\leq s\leq l,$ \eqref{profs}-\eqref{akjot} imply the existence of a positive integer $k_{1}\in {\mathbb N}$ such that, for every $k\geq k_{1},$ we have $
\widehat {G_{\gamma_{l,j}}}(\lambda_{k}^{1},...,\lambda_{k}^{n})=0.$ Therefore, $G_{\gamma_{l,j}}=0$ a.e, so that \eqref{akjot} and the uniqueness theorem for Laplace transform imply  
$G\equiv 0 ,$ and therefore, $f=0$ a.e., as required.
\end{proof}

The following corollary of Theorem \ref{ressuba} and Theorem \ref{ibeer}(iii) seems to be new even in  the one-dimensional setting:

\begin{cor}\label{ojha}
Suppose that $(\lambda_{k})_{k\in {\mathbb N}}$ is a uniqueness sequence for Laplace transform, $\Re \lambda_{k}>0$ for all $k\in {\mathbb N},$ and $\gamma \in (0,1).$ If $\lim_{k\rightarrow +\infty} \Re \lambda_{k}=+\infty,$ then $(\lambda_{k}^{\gamma})_{k\in {\mathbb N}}$ is a uniqueness sequence for Laplace transform and 
\begin{align*} 
\sum_{k=1}^{+\infty}\Biggl[ 1-\Biggl| \frac{  \lambda_{k}^{\gamma}-1 }{  \lambda_{k}^{\gamma}+1 }\Biggr|\Biggr]=+\infty.
\end{align*}
\end{cor}

We would like to make the following observation in connection with Corollary \ref{ojha}:

\begin{rem}\label{apsol}
It is worth noting that Corollary \ref{ojha} cannot be deduced from Theorem \ref{ibeer}(iii) and the implication
\begin{align}\label{naiz}
\sum_{k=1}^{+\infty}\Biggl[ 1-\Biggl| \frac{  \lambda_{k}-1 }{  \lambda_{k}+1 }\Biggr|\Biggr]=+\infty \ \ \Rightarrow \ \ 
\sum_{k=1}^{+\infty}\Biggl[ 1-\Biggl| \frac{  \lambda_{k}^{\gamma}-1 }{  \lambda_{k}^{\gamma}+1 }\Biggr|\Biggr]=+\infty,
\end{align}
since the requirements of Corollary \ref{ojha} do not imply the equality \eqref{ret}, in general. This approach can be used only in the case that there exists a number $\theta \in (0,\pi/2)$ such that $|\arg(\lambda_{k})|\leq \theta,$ $k\in {\mathbb N}.$

We want also to note that the implication \eqref{naiz} is always true if $\Re \lambda_{k}>0,$ $k\in {\mathbb N};$ in actual fact, we have 
\begin{align}\label{subp}
\Biggl| \frac{  \lambda ^{\gamma}-1 }{  \lambda ^{\gamma}+1 }\Biggr| \leq \Biggl| \frac{  \lambda  -1 }{  \lambda  +1 }\Biggr| ,\quad \Re \lambda>0,
\end{align}
with the equality if and only if $\lambda=1.$ In order to prove \eqref{subp}, let us first observe that this inequality holds for positive real numbers since $(\lambda  -1)/(\lambda  +1)=1-2/(\lambda  +1),$ $\lambda>0,$ as well as that the equality in this case holds only for $\lambda=1.$ Suppose now that $\lambda =re^{i\varphi},$ where $r>0$ and $\varphi \in (-\pi/2,\pi/2) \setminus \{0\}.$ Then \eqref{subp} is equivalent with the inequality
$$
\bigl| \lambda -1\bigr| \cdot \bigl| \lambda^{\gamma}+1 \bigr| \geq \bigl| \lambda +1\bigr| \cdot \bigl| \lambda^{\gamma}-1 \bigr|.
$$
After a tedious computation, we obtain that \eqref{subp} is equivalent with the inequality
\begin{align}\label{ocv}
r^{1+\gamma}\cos (\gamma \varphi)+r^{\gamma-1}\cos (\gamma \varphi) \geq r^{2\gamma}\cos \varphi +\cos \varphi.
\end{align}
Since $\cos (\gamma \varphi)>\cos ( \varphi)>0,$ it suffices to show that 
$
r^{1+\gamma} +r^{\gamma-1} \geq r^{2\gamma}  +1, 
$ i.e., $
r^{2+\gamma} +r^{\gamma } \geq r^{2\gamma+1}  +r.
$ But, 
\begin{align*}
r^{2+\gamma} +r^{\gamma } - r^{2\gamma+1}  -r=r\Bigl( r^{1+\gamma}-1\Bigr) \cdot \Bigl(1-r^{\gamma-1}\Bigr) \geq 0,
\end{align*}
which follows by considering separately the cases $0< r\leq 1$ and $r>1.$
\end{rem}

In the existing literature, it has still not been clarified any
inversion formula for Laplace transform that holds for arbitrary uniqueness sequences, even in the one-dimensional setting. We close the paper with the observation that 
the numerical inversion of multidimensional Laplace transform has been considered in many research articles (cf. the list of references quoted in \cite{mvlt}) as well as that the Post-Widder theorem can be simply clarified for the multidimensional vector-valued Laplace transform, which is no longer case for the Phragm\'en–Doetsch inversion (cf. \cite[Theorem 1.4.4]{FKP}) and the 
Phragm\'en-Mikusi\'nski inversion (cf.
\cite[Corollary 1.5]{b3}).

\end{document}